\documentclass{amsart}

\usepackage{amsthm}
\usepackage{amsmath,amssymb}
\usepackage{scalerel}
\usepackage{graphicx}
\usepackage{titlesec}
\usepackage{esint}
\usepackage{calrsfs}
\newtheorem{theorem}{Theorem}

\newtheorem{lem}[theorem]{Lemma}
\newtheorem{definition}[theorem]{Definition}

\numberwithin{theorem}{section}

\allowdisplaybreaks

\newcommand{\be} {\begin{equation}}
\newcommand{\ee} {\end{equation}}
\newcommand{\bea} {\begin{eqnarray}}
\newcommand{\eea} {\end{eqnarray}}
\newcommand{\Bea} {\begin{eqnarray*}}
\newcommand{\Eea} {\end{eqnarray*}}

\makeatletter
\def\specialsection{\@startsection{section}{1}%
  \z@{\linespacing\@plus\linespacing}{.5\linespacing}%
  {\normalfont\centering}% DELETED
  {\normalfont}}% NEW
\def\section{\@startsection{section}{1}%
  \z@{.7\linespacing\@plus\linespacing}{.5\linespacing}%
  {\normalfont\scshape\centering}}% DELETED
  {\normalfont\scshape}% NEW
\makeatother

\numberwithin{equation}{section}

\begin{document}

%\title[Boundary differentiability for strong solutions of elliptic problems on convex domains]
%{Boundary differentiability for strong solutions of elliptic problems on convex domains}
\title[Boundary differentiability etc]
{Boundary differentiability of solutions to elliptic equations in convex domains in the borderline case}
\author[  Dharmendr \;Kumar]
{   Dharmendra\;Kumar  }
%{DHARMENDRA KUMAR}
 \address{K.\,Dharmendra \hfill\break
 Tata Institute of Fundamental Research, Centre For Applicable Mathematics, Post Bag No 6503, GKVK Post Office, Sharada Nagar, Chikkabommsandra, Bangalore 560065, Karnataka, India}
\email{dharmendra2020@tifrbng.res.in, dharamsambey90@gmail.com} 
 
\date{\today}
\thanks{Submitted \today.  Published - - - - -.}
\begin{abstract}
In this work, we consider the following  elliptic partial differential equations:
\begin{equation*}
\left\{
\begin{aligned}{}
  - b_{ij} \; \frac{\partial^{2} w}{\partial x_{i} \partial x_{j}}    &=  g
 \;\;\; \text{in} \;\; \Omega, \\
  w &= 0 \;\;\;\text{on} \;\partial \Omega, 
\end{aligned}
\right.
\end{equation*}
\noindent where the domain $\Omega \subset \mathbb{R}^{n}$ is convex, the matrix $\big(b_{ij}\big)_{n \times n}$  satisfies the uniform ellipticity conditions. For $g$ in the scaling critical  Lorentz space  $ L(n,\; 1)(\Omega)$,  we establish boundary differentiability of solutions  to the above problem.  We also prove  $C^{\mathrm{Log-Lip}}$ regularity estimate at a boundary point  in the case when  $g \in L^{n}(\Omega)$.

\vspace{4mm}

\end{abstract}

\maketitle
\tableofcontents
\section{Introduction}

\noindent  In this work, we consider the following elliptic PDEs:

\begin{equation}\label{20MAR}
\left\{
\begin{aligned}{}
  - b_{ij} \; \frac{\partial^{2} w}{\partial x_{i} \partial x_{j}}    &=  g
 \;\;\; \text{in} \;\; \Omega, \\
  w &= 0 \;\;\;\text{on} \;\partial \Omega, 
\end{aligned}
\right.
\end{equation}

\noindent where the domain $\Omega \subset \mathbb{R}^{n}$, the matrix $\big(b_{ij}\big)_{n \times n}$   is symmetric and 

\begin{equation}\label{20MAR1}
\lambda I \; \leq \; \big(b_{ij}\big)_{n \times n} \; \leq \; \frac{1}{\lambda} I,\;\; \text{for some constant}\; 0 < \lambda  \leq 1.
\end{equation}

\vspace{2mm}

\noindent The aim of this paper is to establish boundary differentiability of strong  solution to \eqref{20MAR} for $g$ in the scaling critical Lorentz space  $L(n,\;1)(\Omega)$.

%\vspace{2mm}

%\noindent Here
%$g \in L(n,\; 1)(\Omega).$ 
%let us notice that 
%$$L^{p} \; \subset\; L(n,\;1)\; \subset\; L^{n}\;\;\;\mathrm{for \; every}\; p > n.$$
%

% This follows from merging the following two basic analysis:
%
%\vspace{1mm}
%
%\noindent - \;\; Krylov and Safanov \cite{KrylovandSafanov1, KrylovandSafanov1} and  
%
%
%\vspace{1mm}
%
%\noindent - \;\; Fabes and Stroock \cite{FabesandStroock}: who established suitable reverse H\"{o}lder inequalities for the fundamental solutions to certain linear elliptic equations.

\vspace{2mm}

\noindent To put our results in the right perspective, we begin with a classical result of  Stein\cite{SteinTheorem} on the  following `` limiting  '' case of Sobolev embedding theorem.

\vspace{2mm}

\begin{theorem}\label{SteinTheorem}
Let $L(n,\;1)$ denote the standard Lorentz space. Then for each Sobolev function $w \in W^{1,1},$ the following implication holds: 
$$\nabla w \;\in \;L(n,\;1)\;\;\Longrightarrow\;\; w \; \mathrm{is \; continuous}.$$
\end{theorem}
\noindent The Lorentz space $L(n,\;1)$ appearing in Theorem \ref{SteinTheorem} is  defined in section 3, see \eqref{26May1} next, for the precise definitions.

\vspace{1mm}

\noindent Theorem \ref{SteinTheorem} can be considered as the limiting case of Sobolev-Morrey embedding  claiming that: 
$$\nabla w \; \in \; L^{n + \varepsilon}\;\; \Longrightarrow\;\; w \;\in\; C^{0, \frac{\varepsilon}{n + \varepsilon}}.$$ 

\vspace{1mm}

\noindent Note that
$$L^{n + \varepsilon}\;\; \subset \;\; L(n,\;1)\;\;\subset\;\;  L^{n}, \;\;\;\; \forall\; \varepsilon > 0,$$
with all the inclusions being strict.
 
\vspace{1mm}

\noindent Now  Theorem \ref{SteinTheorem}  combined with the standard Calder\'{o}n-Zygmond theory has the following interesting consequence: 

\vspace{1mm}

\begin{theorem}\label{CorollaryOfSteinTheorem}
$$\Delta w \; \in \; L(n,\;1)\;\; \Longrightarrow\;\; \nabla w \; \mathrm{is \; continous.}$$
\end{theorem}

\vspace{1mm}

\noindent To the best of our knowledge, we are not aware of analogue of Theorem \ref{CorollaryOfSteinTheorem} for general nonlinear elliptic and parabolic equations occurring in the past through a rather sophisticated and powerful nonlinear potential theory, see for instances \cite{Mingione1,Mingione2,Mingione3} and the references therein.

\vspace{1mm}

\noindent  Theorem \ref{CorollaryOfSteinTheorem} has been extended to quaslinear operators modelled on the $p-$laplacian by Kuusi and Minigione in \cite{Mingione4}. Subsequently they generalized their result to $p$-laplacian type systems in \cite{Mingione5}. See also  \cite{Mingione1,Mingione2,Mingione3} and the references therein.   After that there has been several generalization of such a result to operators with different type of nonlinearities.
 Theorem \ref{CorollaryOfSteinTheorem}   has been extended to fully nonlinear elliptic operators  \cite{TuomoKuusi, AdimurthiAndBanerjee}. Finally we also refer to a recent work \cite{BMu} where the borderline gradient continuity has been obtained  in the case of game-theoretic normalized $p$-Laplacian operator.

\vspace{2mm}

\noindent We now mention some other  works   in the case of boundary regularity that are closely related to this paper.  In 1984, Krylov in \cite{Krylov1,Krylov2} proved that solution of \eqref{20MAR} is $C^{1, \alpha}$ along the boundary under the assumption that $\partial \Omega$ is $C^{1, \alpha}$ using boundary harnack principle.  Such  $C^{1}$ results were extended to $C^{1,Dini}$ domains in  \cite{Lieberman}.   Now in case of convex domains, we see that standard barrier arguments imply that solutions are Lipschitz continuous up to the boundary. We also mention that Wang in \cite{Wang}  established similar results as in \cite{Krylov2}  by a somewhat different iterative argument based on ABP type comparison principle. In 2006, Li and Wang \cite{LiandWang1} adapted such an approach  to establish  the  boundary differentiability of the solution to \eqref{20MAR}  in convex domains. In their work, they assumed that  $g$ satisfies
\begin{equation}\label{c1}\int_{0}^{r_{0}} \frac{\|g\|_{L^{n}\big(\Omega \cap Q[\rho \times \rho]\big)}}{\rho} d\rho\;\; < \;\; \infty. \end{equation}
 For nonhomogeneous Dirichlet boundary condition, see \cite{LiandWang2}.  Also they showed that their result is  optimal in the sense that  only under the assumption that $\Omega$ is convex,  continuity of the gradient of the solution to \eqref{20MAR} along the boundary can not be expected.  They demonstrated this by providing two counterexamples.  In this note, we sharpen the result in \cite{LiandWang2}  to the case when $g \in L(n,1)$.  It is to be noted that  an arbitrary $g \in L(n,1)$ need not satisfy \eqref{c1} and thus our result is not covered by \cite{LiandWang2}. 
 
 \vspace{2mm}
 
 \noindent In closing,  we mention that  Ma and Wang in \cite{MaandWang} established the gradient continuity of solutions  up to the boundary to fully nonlinear uniformly elliptic equations on $C^{1, \mathrm{Dini}}$ domain $\Omega.$ Their result was extended   by Adimurthi and Banerjee in  \cite{AdimurthiAndBanerjee} to the borderline case when the right hand side $g \in L(n,1)$. In fact,  the main result of \cite[Theorem 1.3]{AdimurthiAndBanerjee} can also be thought of as boundary counterpart of \cite[Theorem 1.3]{TuomoKuusi}. 
 
\vspace{2mm}

\vspace{2mm} For  $g \in L^{n}(\Omega)$,  we also  establish Log-Lipschitz estimate at a boundary point  of a convex domain using  compactness arguments  which are inspired by the  works of Caffarelli  in \cite{Caffarelli1, Caffarelli2, Caffarelli3}.

 \vspace{4mm}
 
\noindent \textbf{Notations}:
 
\vspace{2mm}

\noindent $B_{r}\;:=\;$ the open ball of radius $r$ and center $0.$

\vspace{2mm}

\noindent $B_{r}(y) \;:= \;y + B_{r}. $

\vspace{2mm}

\noindent $\mathbb{R}_{+}^{n} \;:= \;\big\{(x',\;x_{n})\;\;: \;\;x' \in \mathbb{R}^{n-1},\; x_{n} > 0\big\}.$

\vspace{2mm}

\noindent $T_{r} \;:= \;B_{r}\; \cap\; \partial{\mathbb{R}_{+}^{n}}.$

\vspace{2mm}

\noindent $Q[c \times d] \;:= \; T_{c}\; \times \;(0,\;d)\; \subset\; \mathbb{R}^{n}$\; for  any real number $c > 0,\; d > 0.$

\vspace{2mm}

\noindent $\fint_{\Omega} f(x)dx :=$ Integral average of $f$ over the positive measure set $\Omega$. 

\vspace{3mm}

\noindent The organization of the paper is as follows. In section 2, we state our main result. In section 3, we shall recall some definitions and characterization of relevant function spaces that will be used in the subsequent sections. Now given t that  the blow-up of a convex set at each boundary point is cone, in section 4,  we first  examine the boundary differentiability of the solutions to \eqref{20MAR}, when $\Omega$ is a  cone.  This is starting point of the proof of Theorem \ref{Mainthm}.  In the same section, we also  establish a basic comparison lemma  that is relevant to the proof of Theorem \ref{22JANTHM} below. In section 5, we prove Theorem \ref{Mainthm}. Furthermore  by compactness arguments, we also prove  Theorem \ref{22JANTHM}.

\section{Main results} 

\noindent Main results of this paper are the following:

\begin{theorem}\label{Mainthm} Assume that  $\Omega$ is convex and $g \in L(n,\;1)(\Omega).$  Let $w$ be a strong  solution of \eqref{20MAR} for $q \in (n - n_{0},\; n),$
where $n_{0}$ is a small universal constant depends only on $n$ and ellipticity constants.  Then $w$ is differentiable at each $x \in \partial \Omega.$

\end{theorem}

\begin{theorem}\label{22JANTHM}
Assume that $\Omega$ is convex and $g \in L^{n}(\Omega)$. Let $w$ be a strong   solution of \eqref{20MAR}, then $w$ is Log-Lipschitz at any 
$x \in \partial \Omega.$
\end{theorem}

\section{Preliminary results}

\noindent In this section, we shall recall some definitions and characterization of relevant function spaces that will be used in the subsequent sections.

\vspace{2mm}
%
%\noindent Let us recall definition of cone.

\begin{definition}
A set $X \subset \mathbb{R}^{n}$ is called a cone if the following two conditions are satisfied: 
\vspace{1mm}

$\mathrm{(i)}$ $X + X \subset X$ 

\vspace{1mm}

$\mathrm{(ii)}$  $\mathbb{R}_{+} \cdot X \subset X.$ 

\end{definition}
\vspace{2mm}

%\begin{definition}
%A continuous function $w$ in $\Omega$ is a viscosity subsolution $($resp. supersolution$)$ of \eqref{20MAR} if for any $x_{0} \in \Omega$ and any function $\varphi \in C^{2}(\Omega)$ such that $w - \varphi$ has a local maximum  $($resp.     
%minimum$)$ at $x_{0}$ implies that 
%$$   - b_{ij}\; \frac{\partial^{2} \varphi(x_{0})}{\partial x_{i} \partial x_{j}}\; \leq\; g(x_{0})\;\;\;\bigg(\mathrm{resp.}\; - b_{ij}\; \frac{\partial^{2} \varphi(x_{0})}{\partial x_{i} \partial x_{j}}\; \geq\; g(x_{0}) \bigg).$$
%
%
%
%
%\end{definition}
%
%\vspace{2mm}
%
%\begin{definition}
%Let $g \in L_{loc}^{p}$ for some $p > \frac{n}{2}.$  A continuous function $w$ is an $W^{2,p}$- viscosity subsolution $($resp. supersolution$)$ of \eqref{20MAR} if $\varphi \in W^{2, p}\big(B_{r}(x_{0})\big)$ with $B_{r}(x_{0}) \subset \Omega$ and for each $\varepsilon > 0$ satisfying the following:
%$$ - b_{ij}\; \frac{\partial^{2} \varphi(x)}{\partial x_{i} \partial x_{j}}\; \geq\; g(x) + \varepsilon )\;\;\;\;\mathrm{almost \; everywhere\; in }\; B_{r}(x_{0} $$
%$$ \bigg(- b_{ij}\; \frac{\partial^{2} \varphi(x)}{\partial x_{i} \partial x_{j}}\; \leq\; g(x) - \varepsilon\;\;\;\;\mathrm{almost \; everywhere\; in }\; B_{r}(x_{0})\bigg),$$
%implies that $w - \varphi$ cannot attain a local maximum   $($minimum$)$ at $x_{0}$.
% 
%\vspace{1mm}
%
%\noindent Moreover, the function $w$ is called $W^{2,p}$- viscosity solution if $w$ is both a subsolution and supersolution.
%
%\end{definition}

\vspace{2mm}

\noindent Relevant Function Spaces: 

\vspace{2mm}

\noindent
Let us also recall some basic maximal-type characterization of Lorentz spaces, see for instantce \cite{Grafakos, TuomoKuusi}.
 
\vspace{2mm}

\noindent Let $\Omega \; \subseteq \; \mathbb{R}^{n}$ be an open set and $1  \leq p  < \infty, \; 0  < q \leq \infty$.

\vspace{2mm}

\noindent The Lorentz space $L(p, \; q)(\Omega)$ is defined by prescribing all those measurable functions $f$ satisfying 
\begin{equation}\label{26May1}
\int_{0}^{\infty}\; \bigg(\lambda^{p} \big|\{x \in \Omega\;\;:\;\; |f(x)| > \lambda\}\big|\bigg)^{q/p}\;\frac{d\lambda}{\lambda}\; < \; \infty\;\;\;\mathrm{when}\; q < \infty;
\end{equation}
and 
\begin{equation}
\sup_{\lambda > 0}\;\lambda^{p} \;\big|\{x \in \Omega\;\;:\;\; |f(x)| > \lambda\}\big| \; \leq \; \infty,\;\;\;\mathrm{when}\; q = \infty.
\end{equation}

\noindent We denote $L(p, \; \infty)(\Omega) \equiv \mathcal{M}^{p}(\Omega).$ This is  known as Marcinkiewicz space. 
 Using above definition, it follows that 
\begin{equation}
f \in L(p,\; q)\; \; \Longrightarrow \;\; |f|^{r} \in L\big(p/r,\; q/r \big)\;\; \mathrm{for}\;\; r \leq p.
\end{equation}

\vspace{2mm}

\noindent First  we introduce the non-increasing rearrangement $$f^{*} : [0,\; \infty) \longrightarrow \; [0,\; \infty)$$ of a measurable function $f$ defined by
$$f^{*}(s) \; := \; \sup\Big\{t \geq 0\;\; :\;\; \big|\{x \in \mathbb{R}^{n}\;\;: |f(x)| > t\}\big| > s\Big\}.$$ And set 
$$f^{**}(\varrho) \;:=\; \varrho^{-1}\; \int_{0}^{\varrho}f^{*}(t)dt.$$

\vspace{2mm}

\noindent For $p > 1, \; q > 0$, maximal-type characterization of $L(p,\;q)$  claims that 
$$f \in L(p,\;q) \;\;\;\Longleftrightarrow\;\; \int_{0}^{\infty}\;\Big(f^{**}(\varrho)\varrho^{1/p}\big)^{q}\;\frac{d\varrho}{\varrho} < \infty.$$

\vspace{2mm}

\noindent $\mathrm{We\ also  \; us \; fix \; an \; exponent}\; q \in (n - n_{0},\; n).$ 
Here $n_{0}$ is a small universal constant depends only on $n$ and ellipticity constants such that the generalized maximum principle as  in  \cite{Escauriaza,CaffarelliCrandallKocanSwiech} holds for all $q \in (n-n_0, n]$.

\section{Some useful lemmas}

\noindent Since the blow-up of a convex set at any boundary point is cone, we should first consider Theorem \ref{Mainthm}  in the case when the domain is a cone.

\vspace{1mm}

\noindent Assume that $X \subset \mathbb{R}_{+}^{n}$ is an open cone with nonempty interior. For $X \neq \mathbb{R}_{+}^{n}$, we define 

\begin{equation}\label{26Mar}
\nu = \inf\big\{\; r > 0 \; : \; (e_{n} + T_{r}) \cap \partial X \neq \emptyset\big\}.
\end{equation}

\noindent Then it is easy to see that $\nu < \infty.$ Let $r_{0} > 0$ and $w$ be a solution of the following elliptic PDEs:
\begin{equation}\label{26MAR1}
\left\{
\begin{aligned}{}
  - b_{ij} \; \frac{\partial^{2} w}{\partial x_{i} \partial x_{j}}    &=  g
 \;\;\; \text{in} \;\; X \cap Q[r_{0} \times r_{0}], \\
  w &= 0 \;\;\;\text{on} \;\partial X \cap Q[r_{0} \times r_{0}]. 
\end{aligned}
\right.
\end{equation}

%
%\section{Some useful  Lemma}
%
%\noindent In this section, we will prove the following lemma: 
% which correspond to the two different cases:
%
%\vspace{2mm}
%
%(i)  $X \neq \mathbb{R}_{+}^{n}$ 
%
%(ii) $X = \mathbb{R}_{+}^{n}$.

\begin{lem}\label{MainLemma}
If $X = \mathbb{R}_{+}^{n},$ then $\exists$ positive constants $C, \beta$ and $\Lambda$ depending only on $\lambda, n$ and $\nu$ such that 

\begin{equation}\label{27Mar1}
\begin{aligned}{}
|w(x) - a x_{n}| \leq C &\bigg(\frac{r^{\beta}}{r_{0}^{1 + \beta}} \|w\|_{L^{\infty}\big(X \cap Q[r_{0} \times r_{0}]\big)} + \frac{r^{\beta}}{r_{0}^{\beta}} \|g\|_{L^{q}\big(X \cap Q[r_{0} \times r_{0}]\big)} \\
 & + \; r^{\beta} \int_{r}^{r_{0}} \frac{\|g\|_{L^{q}\big(X \cap Q[\rho \times \rho]\big)}}{\rho^{1 + \beta}} d\rho + \int_{0}^{\Lambda r} \frac{\|g\|_{L^{q}\big(X \cap Q[\rho \times \rho]\big)}}{\rho} d\rho +  \|g\|_{L^{q}\big(X \cap Q[\Lambda r \times \Lambda r]\big)}\bigg)r,
\end{aligned}
\end{equation}
for any $x \in Q[r \times r] \cap X$ and $r \leq \frac{r_{0}}{\Lambda}.$
\end{lem}

\begin{proof}

\vspace{2mm}

\textbf{Claim I}\;$Normalization.$  

\vspace{2mm}

By translation, rotation and scaling, we shall assume that $$r_{0} = 1,\;\;\; \|w\|_{L^{\infty}\big(X \cap Q[1 \times 1]\big)} + \|g\|_{L^{q}\big(X \cap Q[1 \times 1]\big)} \; \leq 1.$$

\noindent Since the equation is linear, without loss of generality, we can assume that $w(x) \geq 0$ for $x \in Q[1 \times 1] \cap X$ and $f(x) \geq 0$ for $x \in Q[1 \times 1] \cap X.$

\vspace{2mm}

\noindent From the above fact, instead of establishing inequality \eqref{27Mar1}, it is enough to show
\begin{equation}\label{30Mar}
\begin{aligned}{}
|w(x)| \leq C &\bigg(r^{\beta} 
  + \; r^{\beta} \int_{r}^{1} \frac{\|g\|_{L^{q}\big(X \cap Q[\rho \times \rho]\big)}}{\rho^{1 + \beta}} d\rho + \int_{0}^{\Lambda r} \frac{\|g\|_{L^{q}\big(X \cap Q[\rho \times \rho]\big)}}{\rho} d\rho + \|g\|_{L^{q}\big(X \cap Q[\Lambda r \times \Lambda r]\big)}\bigg)r,
\end{aligned}
\end{equation}
for any $x \in Q[r \times r] \cap X$ and $r \leq \frac{1}{\Lambda}.$

\vspace{2mm}

\noindent \textbf{Claim II}\;  Next we claim that $\exists$ positive constants $\sigma, \eta (< 1), K_{1}, K_{2}$ and $K_{3}$ depending only on $\lambda$ and $n$ such that whenever
\begin{equation}\label{30Mar1}
m x_{n} - c \leq w(x) \leq M x_{n} + c \;\; \mathrm{for \;any} \;x \in Q[1 \times 1], 
\end{equation}
for some  constants $c\;(>0)$, $m$ and $M$, then this implies that $\exists$ $\tilde{m}$ and $\tilde{M}$ such that
\begin{equation}\label{30Mar2}
\tilde{m}\; x_{n} - K_{1}\|g\|_{L^{q}\big(Q[1 \times 1]\big)} \leq  w(x) \leq \tilde{K}\; x_{n} + K_{1} \|g\|_{L^{q}\big(Q[1 \times 1]\big)}\;\; \mathrm{for\; any} \;x \in Q[\sigma \times \sigma], 
\end{equation}
with 
\begin{equation}\label{9APRIL1}
\tilde{M} - \tilde{m} = (1 - \eta)(M - m) + K_{2}\;c
\end{equation}
and 
\begin{equation}\label{9APRIL2}
\big|(\tilde{M} + \tilde{m}) - (M + m)\big| \leq \eta (M - m) + K_{3}\;c.
\end{equation}

\noindent Proof of the claim: 

\vspace{2mm}

\noindent Define $K := \sqrt{n - 1}\bigg(1 + \frac{2 \sqrt{n - 1}}{\lambda}\bigg),\; $ and choose $\varepsilon > 0$ such that 
\begin{equation}\label{30Mar3}
(1 + \varepsilon)(2 + \varepsilon)(K - 1)^{\varepsilon} \leq 4.
\end{equation}

\noindent Let 
\begin{equation}\label{30Mar4}
\varphi(x) = \frac{2x_{n}}{\tilde{\sigma}} - \bigg(\frac{x_{n}}{\tilde{\sigma}}\bigg)^{2} + \frac{\lambda^{2}}{2(n - 1)} \sum_{i = 1}^{n-1}\Big(\big(\big|\frac{x_{i}}{\tilde{\sigma}}\big| - 1\big)^{+}\Big)^{2 + \varepsilon},
\end{equation}
where $\tilde{\sigma} = \frac{1}{K}.$

\vspace{2mm}

\noindent Then $\varphi$ is differentiable and we have the following: 

\begin{equation}\label{30Mar5}
\left\{
\begin{aligned}{}
  - b_{ij} \; \frac{\partial^{2} \varphi}{\partial x_{i} \partial x_{j}}    &\geq  0
 \;\;\; \text{in} \;\;  Q[1 \times \tilde{\sigma}], \\
  \varphi(x) &\geq  1 \;\;\;\text{if} \;\;x_{n} = \tilde{\sigma}; \; \text{or}\; |x'| = 1\;\;\text{and}\; 0 \leq x_{n} \leq \tilde{\sigma};\\
  \varphi(x) &\geq  0 \;\;\;\text{if} \;\;x_{n} = 0.
\end{aligned}
\right.
\end{equation}

\noindent From \eqref{26MAR1} and \eqref{30Mar5}, we get
\begin{equation}\label{30Mar6}
\left\{
\begin{aligned}{}
  - b_{ij} \; \frac{\partial^{2}}{\partial x_{i} \partial x_{j}}\big(w(x) - M x_{n} - c\varphi(x)\big)    &\leq  g
 \;\;\; \text{in} \;\;  Q[1 \times \tilde{\sigma}], \\
  w(x) - M x_{n} - c\varphi(x) &\leq  0 \;\;\;\text{on} \;\;\partial( Q[1 \times \tilde{\sigma}]).
\end{aligned}
\right.
\end{equation}
By the generalized maximum principle \cite{Escauriaza,CaffarelliCrandallKocanSwiech} for $q \in (n - n_{0},\; n),$
where $n_{0}$ is a small universal constant depends only on $n$ and ellipticity constants, we have
\begin{equation}\label{30Mar7}
 w(x) - M x_{n} - c\varphi(x) \leq C \|g\|_{L^{q}(X \cap Q[1 \times \tilde{\sigma}])},\;\; \mathrm{for \; each}\; x \in  Q[1 \times \tilde{\sigma}],
\end{equation}
with a constant $C = C(\lambda, n) > 0.$

\vspace{2mm}

\noindent Since $\varphi(x) \leq \frac{2 x_{n}}{\tilde{\sigma}} = 2 K x_{n}$ for each 
$x \in Q[\tilde{\sigma} \times \tilde{\sigma}],$ this implies that
\begin{equation}\label{30Mar8}
w(x) \leq (M + 2Kc)x_{n} + C \|g\|_{L^{q}(Q[1 \times 1])} \;\; \mathrm{in}\; Q[\tilde{\sigma} \times \tilde{\sigma}].
\end{equation}
\noindent By the same arguments as \eqref{30Mar8}, we have 
\begin{equation*}\label{9APRIL3}
 (m - 2Kc)x_{n} - C \|g\|_{L^{q}(Q[1 \times 1])} \leq w(x) \;\; \mathrm{in}\; Q[\tilde{\sigma} \times \tilde{\sigma}],
\end{equation*}
by considering $mx_{n} - c\varphi(x) - w(x)$ instead of $ w(x) - M x_{n} - c\varphi(x)$. 

\vspace{2mm}

\noindent Hence, we have the following estimate in $Q[\tilde{\sigma} \times \tilde{\sigma}]$
\begin{equation}\label{9APRIL3}
 (m - 2Kc)x_{n} - C \|g\|_{L^{q}(Q[1 \times 1])} \leq w(x) \leq (M + 2Kc)x_{n} + C \|g\|_{L^{q}(Q[1 \times 1])}.
\end{equation}
\noindent We define 
$$ \sigma = \frac{\tilde{\sigma}}{2K} \;\; \mathrm{and}\;\; \Sigma = \big\{\sigma e_{n} + T_{\sigma K}\big\}.$$

\vspace{2mm}

\noindent We shall divide the remaining proof into two cases: $$w(\sigma e_{n}) \geq \frac{1}{2}(M + m)\sigma\;\; \mathrm{and}\;\;w(\sigma e_{n}) < \frac{1}{2}(M + m)\sigma.$$

\vspace{2mm}

\noindent Let us consider the case: $w(\sigma e_{n}) \geq \frac{1}{2}(M + m)\sigma.$
 We define 
 \begin{equation}\label{10APRIL1}
 u(x):= w(x) - (m - 2Kc)x_{n} + C \|g\|_{L^{q}(Q[1 \times 1])}.
 \end{equation}
\noindent From \eqref{9APRIL3}, it is easy to see that 
 \begin{equation}\label{10APRIL2}
u(x) \geq 0, \;\; \mathrm{for \; each}\; x \in Q[\tilde{\sigma} \times \tilde{\sigma}].
 \end{equation}
Then
 \begin{equation}\label{10APRIL3}
u(\sigma e_{n}) \geq \frac{1}{2} (M - m + 4Kc)\sigma + C \|g\|_{L^{q}(Q[1 \times 1])}. 
 \end{equation}
From the Harnack inequality, we conclude that
\begin{equation*}
\inf_{\Sigma} u(x) \geq  \frac{1}{2\tilde{C}}(M - m + 4Kc)\sigma +  C\|g\|_{L^{q}(X \cap Q[1 \times 1)},
\end{equation*}
with constant $\tilde{C} = \tilde{C}(\lambda, n).$

\noindent Set 
$$ \mathcal{M} = \Bigg(\frac{1}{2\tilde{C}}(M - m + 4Kc)\sigma + C \|g\|_{L^{q}(X \cap Q[1 \times 1])}\Bigg)^{+}. $$
Hence it follows that 
\begin{equation}\label{10APRIL4}
\inf_{\Sigma} u(x) \geq \mathcal{M}.
\end{equation}

\noindent Now we consider the following barriers
\begin{equation}\label{3APR2}
\Phi(x) = \frac{1}{2} \Bigg(\frac{x_{n}}{\sigma} + \Big(\frac{x_{n}}{\sigma}\Big)^{2}\Bigg) - \frac{\lambda^{2}}{4(n - 1)} \sum_{i = 1}^{n - 1} \Bigg(\Big(\big|\frac{x_{i}}{\sigma}\big| - 1\Big)^{+}\Bigg)^{2 + \varepsilon}.
\end{equation}
\noindent Then $\Phi$ is differentiable with the followings:
\begin{equation*}
\left\{
\begin{aligned}{}
  - b_{ij} \; \frac{\partial^{2} \Phi}{\partial x_{i} \partial x_{j}}    &\leq  0
 \;\;\; \text{in} \;\;  Q[K \sigma \times \sigma], \\
  \Phi(x) &\leq  1 \;\;\;\text{if} \;\;x_{n} = \sigma;\\
  \Phi(x) &\leq  0 \;\;\;\text{if} \;\;x_{n} = 0; \; \text{or}\; |x'| = \frac{\tilde{\sigma}}{2} = K \sigma\;\;\text{and}\; 0 \leq x_{n} \leq \sigma.
\end{aligned}
\right.
\end{equation*}
\noindent From \eqref{26MAR1}, \eqref{10APRIL2} and \eqref{10APRIL4}, we see that
\begin{equation*}
\left\{
\begin{aligned}{}
  - b_{ij} \; \frac{\partial^{2}}{\partial x_{i} \partial x_{j}}\big(\mathcal{M}\Phi - u \big)    &\leq  g
 \;\;\; \text{in} \;\;  Q[K\sigma \times \sigma], \\
  \mathcal{M}\Phi - u &\leq  0 \;\;\;\text{on} \;\;\partial( Q[K\sigma \times \sigma]).
\end{aligned}
\right.
\end{equation*}
\noindent By the generalized maximum principle \cite{Escauriaza,CaffarelliCrandallKocanSwiech} for $q \in (n - n_{0},\; n),$ we have
\begin{equation*}
\mathcal{M} \Phi(x) - u(x)  \leq C \|g\|_{L^{q}(X \cap Q[K\sigma \times \sigma])},\;\; \mathrm{for \; each}\; x \in  Q[K\sigma \times \sigma],
\end{equation*}
with a constant $C = C(\lambda, n, \nu).$

\noindent Furthermore making use of $\Phi(x) \geq \frac{1}{2\sigma} x_{n}$ in $Q[\sigma \times \sigma],$ we get 
\begin{equation}\label{3APR3}
\frac{\mathcal{M}}{2\sigma} x_{n} - u(x) \leq C \|g\|_{L^{q}(X \cap Q[K\sigma \times \sigma])},\;\; \mathrm{for \; each}\; x \in Q[\sigma \times \sigma].
\end{equation}

\noindent Thus using \eqref{10APRIL1} and the second inequality of \eqref{9APRIL3}, we get \eqref{30Mar2}, \eqref{9APRIL1} and \eqref{9APRIL2}.

\vspace{2mm}

\noindent \textbf{Claim III:}

\vspace{2mm}

We will show that  if $l \geq 0$ and 
$$mx_{n} - c \leq w(x) \leq Mx_{n} + c, \;\; \mathrm{for \; each}\; x \in  Q[\sigma^{l} \times \sigma^{l}], $$
\noindent for some  constants $c > 0, M$ and $m$, then $\exists$ constants $\tilde{m}$ and $\tilde{M}$ such that the following bound holds: 
\begin{equation}\label{3APR4}
\tilde{m} x_{n} - K_{1} \sigma^{l}\|g\|_{L^{q}( Q[\sigma^{l} \times \sigma^{l}])} \leq w(x) \leq \tilde{M} x_{n} + K_{1} \sigma^{l}\|g\|_{L^{q}( Q[\sigma^{l} \times \sigma^{l}])}, 
\end{equation}
$ \mathrm{for \; each}\; x \in  Q[\sigma^{l + 1} \times \sigma^{l + 1}],$ where 
\begin{equation}\label{10APRIL5}
\tilde{M} - \tilde{m} = (1 - \eta)(M - m) + K_{2}\frac{c}{\sigma^{m}},
\end{equation}
and
\begin{equation}\label{10APRIL6}
\big|(\tilde{M} + \tilde{m}) - (M + m)\big| \leq \eta(M - m) + K_{3}\frac{c}{\sigma^{m}},
\end{equation}
for some constants $\eta, \sigma, K_{1}, K_{2}$ and $K_{3}$ are as in  \textbf{Claim II}.
\vspace{2mm}

\noindent Proof of the claim: 

\vspace{2mm}
\noindent Set $$x = \sigma^{l} y,\;\;  a_{ij}(y) = b_{ij}(\sigma^{l}y), \;\;u(y) = \frac{w(\sigma^{l}y)}{\sigma^{l}} \;\; \mathrm{and}\;\;  f(y) = \sigma^{l}g(\sigma^{l}y), $$
for each $y \in Q[1 \times 1].$

\vspace{2mm}

\noindent From Claim II, we have the following estimate in $Q[\sigma \times \sigma]$: 
\begin{equation*}
\tilde{m} y_{n} - K_{1}\|f\|_{L^{q}(X \cap Q[1 \times 1])} \leq  u(y) \leq \tilde{M}y_{n} + K_{1} \|f\|_{L^{q}(X \cap Q[1 \times 1])}.
\end{equation*}
\noindent It is easy to see that above inequality is equivalent to \eqref{3APR4}. Also  $\tilde{M} - \tilde{m}$ and $\big|(\tilde{M} + \tilde{m}) - (M + m)\big|$ satisfy  \eqref{10APRIL5} and  \eqref{10APRIL6}.

\vspace{2mm}

\noindent \textbf{Proof of \eqref{30Mar}}: 

$$\mathrm{Set}\; \; g_{r} := \|g\|_{L^{q}(X \cap Q[r \times r])}, \;\;\;\mathrm{for}\; 0 < r \leq 1.$$

\vspace{2mm}

\noindent Starting from $-1 \leq w \leq 1$ in $ Q[1 \times 1]$ and using Claim III repeatedly, we get the following estimate  for $l = 1,2,\cdot\cdot\cdot$ \;in $Q[\sigma^{l} \times \sigma^{l}]:$
\begin{equation}\label{10APRIL7}
m_{l}x_{n} - K_{1}\sigma^{l-1}g_{\sigma^{l-1}} \leq w(x) \leq M_{l}x_{n} + K_{1}\sigma^{l-1}g_{\sigma^{l-1}}, 
\end{equation}

\noindent with
$$M_{l} - m_{l} = (1 - \eta)^{l-1}K_{2} + \frac{K_{1}K_{2}}{\sigma}\sum_{i=0}^{l-2}g_{\sigma^{i}}(1 - \eta)^{l-2-i}$$
and 
$$\big|(M_{l+1} + m_{l+1}) - (M_{l} + m_{l})\big| \leq \eta(M_{l} - m_{l}) + \frac{K_{1}K_{3}f_{\sigma^{l-1}}}{\sigma}.$$

\noindent Let us suppose that $1 - \eta = \sigma^{\gamma}.$ \;Since $g_{1} \leq 1,$ we have
\begin{equation}\label{10APRIL8}
M_{l} - m_{l} \;\leq \;C (\sigma^{l-1})^{\gamma}\;\Bigg(1 + \int_{\sigma^{l-1}}^{1}\;\frac{g_{r}}{r^{1 + \gamma}}\Bigg)
\end{equation}
and 
\begin{equation*}
\begin{aligned}
\sum_{j=l}^{\infty} \big|(M_{j + 1} + m_{j + 1}) - (M_{j} + m_{j})\big| \leq \;\;\;& \eta\; C \sum_{j=l}^{\infty} (\sigma^{l-1})^{\gamma}\;\Bigg(1 + \int_{\sigma^{l-1}}^{1}\;\frac{g_{r}}{r^{1 + \gamma}}\Bigg)\\
&\;\;\;\;\; + \;\;\frac{K_{1}K_{3}}{\sigma} \;\sum_{j=l}^{\infty} g_{\sigma^{l-1}}.
\end{aligned}
\end{equation*}

\vspace{2mm}

\noindent  Set $G_{r} := \int_{r}^{1}\frac{g_{t}}{t^{1 + \gamma}}dt.$ 

\vspace{2mm}

\noindent \textbf{Observation 1:}
\begin{equation*}
\begin{aligned}
\sum_{j=l}^{\infty}(\sigma^{l-1})^{\gamma}\int_{\sigma^{l-1}}^{1}\frac{g_{r}}{r^{1 + \gamma}}dr &= \sum_{j = l-1}^{\infty}\;(\sigma^{l})^{\gamma}\;G_{\sigma^{l}}\;\frac{\sigma^{l} - \sigma^{l + 1}}{\sigma^{l} - \sigma^{l + 1}}\\
& = \frac{1}{(1 - \sigma)\sigma^{\gamma}} \sum_{j = l-1}^{\infty}\; \int_{\sigma^{l+1}}^{\sigma^{l}}r^{\gamma - 1}G_{r}dr\\
& = \frac{1}{(1 - \sigma)\sigma^{\gamma}} \Bigg( \int_{0}^{\sigma^{l-1}}\frac{g_{r}}{r}dr\; + \; (\sigma^{l-1})^{\gamma}\int_{\sigma^{l-1}}^{1}\frac{g_{r}}{r^{1 + \gamma}}dr \Bigg).
\end{aligned}
\end{equation*}

\vspace{2mm}

\noindent \textbf{Observation 2:}
\begin{equation*}
\begin{aligned}
\sum_{j=l}^{\infty}\; g_{\sigma^{l-1}}\; &\leq\; \frac{1}{1 - \sigma}\; \int_{0}^{\sigma^{l-2}}\;\frac{g_{r}}{r}\;dr \\
& \leq \; \frac{1}{1 - \sigma}\; \int_{0}^{\sigma^{l-2}}\;\frac{1}{r} \Bigg(\int_{B_{\sqrt{n}r}}|g|^{q}dx\Bigg)^{1/p}dr\\ 
& \leq \; C\; \int_{0}^{\sigma^{l-2}}\; \Bigg(\fint_{B_{\sqrt{n}r}}|g|^{q}dx\Bigg)^{1/p}dr\\
 & \leq \; C\; \int_{0}^{\sigma^{l-2}}\; \Bigg(\fint_{B_{\rho}}|g|^{q}dx\Bigg)^{1/p}d\rho\\
 &=  C \; \tilde{\mathrm{I}}_{q}^{g}\;\big(0,\;\sigma^{l-2}\;\big),
\end{aligned}
\end{equation*}
\noindent with  constant $C = C(n, \sigma)$.
\vspace{2mm}

\noindent Also 
$$\sum_{j=l}^{\infty}\;(\sigma^{j-1})^{\alpha} = \frac{(\sigma^{m-1})^{\alpha}}{1 - \sigma^{\gamma}}.$$

\vspace{2mm}

\noindent \textbf{Observation 3:}

\vspace{2mm}

\noindent 
\begin{equation}\label{17APRIL1}
\sup_{x}\;\tilde{\mathrm{I}}_{q}^{g}\;\big(x,\; r\;\big)\; \leq \; \frac{1}{|B_{1}|^{\frac{1}{n}}}\; \int_{0}^{|B_{r}|}\;\big[g^{**}(\rho)\;\rho^{\frac{q}{n}}\big]^{\frac{1}{q}}\; \frac{d\rho}{\rho}. 
\end{equation}

\noindent  Since $g \in L(n,\;1)$, then this implies that $|g|^{q} \in L(n/q, \; 1/q).$ \;So right hand side of \eqref{17APRIL1} is finite. Also since right hand side of \eqref{17APRIL1} tends to zero as $r \rightarrow 0,$ it follows  that 
$$\sup_{x}\;\tilde{\mathrm{I}}_{q}^{g}\;\big(x,\; r\;\big)\; \rightarrow\; 0\;\; \mathrm{uniformly \; with\; respect \; to}\; x \; \mathrm{as} \; r\; \rightarrow \; 0. $$

\vspace{2mm}

\noindent \textbf{Observation 4:}

\vspace{2mm}
\begin{equation}\label{17APRIL2}
\begin{aligned}
\sum_{j=l}^{\infty} \big|(M_{l + 1} + m_{l + 1}) - (M_{l} + m_{l}) \big| \leq \;C \Bigg\{(\sigma^{l-1})^{\gamma} + 
(\sigma^{l-1})^{\gamma}\; \int_{\sigma^{l-1}}^{1}\;\frac{g_{r}}{r^{1 + \gamma}} dr \; +  \int_{0}^{\sigma^{l-2}}\frac{g_{r}}{r} dr\Bigg\}.
\end{aligned}
\end{equation}
\noindent Also 
$$(\sigma^{l-1})^{\gamma}\; + \;
(\sigma^{l-1})^{\gamma}\; \int_{\sigma^{l-1}}^{1}\;\frac{g_{r}}{r^{1 + \gamma}}\; dr \; \;+  \;\int_{0}^{\sigma^{l-2}}\frac{g_{r}}{r} \;dr \; \; \longrightarrow \;0\;\;\;\mathrm{as \;}\;\; l \;\rightarrow\; \infty.
$$
\noindent Therefore, $\big(M_{l} + m_{l}\big)$ is convergent, say
$$\lim_{l \rightarrow \infty}\; \frac{\big(M_{l} + m_{l}\big)}{2}\; = \; b.$$

\noindent Now, 
\begin{equation*}
\begin{aligned}
\Bigg|\; b - \frac{\big(M_{l} + m_{l}\big)}{2}\;\Bigg|\;\; &\leq \; \sum_{j=l}^{\infty}\;
\Bigg|\;\frac{\big(M_{l + 1} + m_{l + 1}\big)}{2}\; - \; \frac{\big(M_{l} + m_{l}\big)}{2}\;\Bigg|\\
& \; \leq \;C \Bigg\{(\sigma^{l-1})^{\gamma} + 
(\sigma^{l-1})^{\gamma}\; \int_{\sigma^{l-1}}^{1}\;\frac{g_{r}}{r^{1 + \gamma}} dr \; +  \int_{0}^{\sigma^{l-2}}\frac{g_{r}}{r} dr\Bigg\}.
\end{aligned}
\end{equation*}

\noindent Thus using \eqref{10APRIL8} and \eqref{17APRIL2}, we get the following estimate in $Q[\sigma^{l} \times \sigma^{l}],$
\begin{equation*}
\begin{aligned}
\big|w(x) - b\;x_{n}\big|\; &\leq \; \Big|w(x)\; - \;\frac{\big(M_{l} + m_{l}\big)}{2}x_{n} \Big|\;\; + \;\; \Big|b - \; \;\frac{\big(M_{l} + m_{l}\big)}{2}\Big|x_{n}\\
& \leq\;\frac{\big(M_{l} - m_{l}\big)}{2}\;x_{n} \; + \; K_{1} \sigma^{l - 1}g_{\sigma^{l-1}}\; + \; \Big|b - \; \;\frac{\big(M_{l} + m_{l}\big)}{2}\Big|x_{n}\\
& \; \leq \;C \Bigg\{(\sigma^{l-1})^{\gamma} + 
(\sigma^{l-1})^{\gamma}\; \int_{\sigma^{l-1}}^{1}\;\frac{g_{r}}{r^{1 + \gamma}} dr \; +  \int_{0}^{\sigma^{l-2}}\frac{g_{r}}{r} dr\Bigg\}x_{n}\; + \; K_{1}\;\sigma^{l-1} f_{\sigma^{l-1}}.
\end{aligned}
\end{equation*}
\noindent This completes the proof of the Lemma \eqref{MainLemma}.

%$$w(x) \leq (1 - \eta)^{l} \Bigg(2K + \frac{2KB}{\sigma} \sum_{i=0}^{l-2}\frac{f_{\sigma^{i}}}{(1 - \eta)^{1 + i}}\Bigg)x_{n} + B \sigma^{l - 1}g_{\sigma^{l-1}}.$$
%\vspace{2mm}
%
%\noindent Suppose that $\sigma^{\theta} = 1 - \eta.$ For $l \geq 2,$ we have
%\begin{equation*}
%\left\{
%\begin{aligned}{}
%\sum_{i = 1}^{l-2} \frac{g_{\sigma^{i}}}{(1 - \eta)^{1 + i}} = \frac{1}{(1 - \sigma)\sigma^{2\theta}} \; \sum_{i=1}^{l-2} \frac{\sigma^{i}}{\sigma^{(i-1)(1 + \theta)}}(\sigma^{i-1} - \sigma^{i})
%\end{aligned}
%\right.
%\end{equation*}
\end{proof}
\noindent  Let us state an important approximation  lemma, which will be useful in the proof of the Theorem \ref{22JANTHM}. 

\begin{lem}
\noindent Given $\varepsilon > 0, \;\; \exists\;  \delta > 0$ such that 
$$\|g\|_{L^{n}(\Omega)} \leq \delta.$$

\noindent Then $\exists$ a function  $u \in C^{0}(B_{1})$ such that 
\begin{equation*}\label{22JAN1}
\left\{
\begin{aligned}{}
  - b_{ij} \; \frac{\partial^{2} u}{\partial x_{i} \partial x_{j}}   &=  0
 \;\;\; \text{in} \;\; \Omega \cap B_{1}, \\
  u &= 0 \;\;\;\text{on} \;\partial \Omega \cap B_{1}, 
\end{aligned}
\right.
\end{equation*}
and 
\begin{equation}\label{SAT23JAN0}
\|w - u\|_{L^{\infty} (B_{1} \cap\; \Omega )} \leq \varepsilon,
\end{equation}
 where $w$ is the strong solution of  \eqref{20MAR}.
\end{lem}

\vspace{2mm}

\begin{proof}
\noindent Let $v$ be the solution of the following equations

\begin{equation*}\label{6FEB}
\left\{
\begin{aligned}{}
  -b_{ij} \; \frac{\partial^{2} u}{\partial x_{i} \partial x_{j}}   &=  0
 \;\;\; \text{in} \;\; B_{1} \cap \Omega, \\
  v &= 0 \;\;\;\text{on} \;\partial \big(B_{1}  \cap \Omega \big)  \setminus \Omega, \\
  v &= u \;\;\;\text{on} \; \partial \big(B_{1}  \cap \Omega \big)  \cap \Omega.
\end{aligned}
\right.
\end{equation*}

\noindent Then from \eqref{20MAR}, we have

\begin{equation*}\label{7FEB1}
\left\{
\begin{aligned}{}
  -b_{ij} \; \frac{\partial^{2} (w - u)}{\partial x_{i} \partial x_{j}}   &=  g
 \;\;\; \text{in} \;\; B_{1} \cap \Omega, \\
  w-u &= 0 \;\;\;\text{on} \;\partial \big(B_{1}  \cap \Omega \big).
\end{aligned}
\right.
\end{equation*}

\noindent By using the Alexandroff - Bakelman - Pucci maximum principle \cite{Caffarelli and Cabre, Gilbarg and Trudinger}, we have

\begin{equation*}\label{7FEB2}
\|w - u\|_{L^{\infty} (B_{1} \cap\; \Omega )} \leq C \;\|g\|_{L^{n} (B_{1} \cap \;\Omega )},
\end{equation*}

\noindent where $C$ is a constant depending only on $n$.

\noindent By choosing $ \delta = \frac{\varepsilon}{C},$ we get 

\vspace{2mm}

\begin{equation*}\label{}
\|u - v\|_{L^{\infty} (B_{1} \cap\; \Omega )} \leq \varepsilon.
\end{equation*}

\end{proof}

\section{Proof of the Theorem \ref{Mainthm}}
\noindent By translation, rotation and scaling, we can assume that $0 \in \partial \Omega$. Hence we only consider the differentiability at the origin. Since $\Omega$ is convex, without loss of generality, we can assume that $\Omega \subset \mathbb{R}_{+}^{n}.$ We will consider the proof of the Theorem \ref{Mainthm} into two cases:
$$ X \neq \mathbb{R}_{+}^{n}\;\;\; \mathrm{and}\;\;\; X = \mathbb{R}_{+}^{n}.$$

\vspace{2mm}

\noindent Case I: \;\; $ X \neq \mathbb{R}_{+}^{n}.$

\vspace{2mm}

\noindent The proof of this case is unchanged as in \cite[ Setion 3, Proof of Theorem 1.1]{LiandWang1}.

\vspace{2mm}

\noindent Case II: \;\; $ X = \mathbb{R}_{+}^{n}.$

\vspace{2mm}

\noindent Let us choose positive sequence $\{r_{j}\}_{j=0}^{\infty}$. Also suppose that $w_{j}$ be the solution of the following problem: 
\begin{equation}\label{}
\left\{
\begin{aligned}{}
  - b_{ij} \; \frac{\partial^{2} w_{j}}{\partial x_{i} \partial x_{j}}    &=  g
 \;\;\; \text{in} \;\;  Q[r_{j} \times r_{j}], \\
  w_{j} &= 0 \;\;\;\text{on} \;\partial Q[r_{j} \times r_{j}] \setminus \Omega,\\ 
   w_{j} &= w \;\;\;\text{on} \;\partial Q[r_{j} \times r_{j}] \cap \Omega.
\end{aligned}
\right.
\end{equation}

\vspace{2mm}

\noindent By invoking Lemma \ref{MainLemma},\; $\exists$ constants $C, \; \Lambda, \;\beta$ depending only on $\lambda, n$ and $\exists$ constants $a_{j}$ such that 
\begin{equation}\label{12May1}
|w_{j}(x) - a_{j} x_{n}| \;\; \leq \;\; C\; G_{j}(r)r,\;\; \forall \; x \in Q [r \times r]\; \mathrm{and}\; r \leq \frac{r_{j}}{\Lambda},
\end{equation}
\noindent where 
\begin{equation}\label{12May2}
\begin{aligned}{}
G_{j}(r) =   &\frac{r^{\beta}}{r_{j}^{1 + \beta}} \|w\|_{L^{\infty}\big(X \cap Q[r_{j} \times r_{j}]\big)} + \frac{r^{\beta}}{r_{j}^{\beta}} \|g\|_{L^{q}\big(X \cap Q[r_{j} \times r_{j}]\big)} \\
 & + \; r^{\beta} \int_{r}^{r_{j}} \frac{\|g\|_{L^{q}\big(X \cap Q[\rho \times \rho]\big)}}{\rho^{1 + \beta}} d\rho + \int_{0}^{\Lambda r} \frac{\|g\|_{L^{q}\big(X \cap Q[\rho \times \rho]\big)}}{\rho} d\rho +  \|g\|_{L^{q}\big(X \cap Q[\Lambda r \times \Lambda r]\big)}.
\end{aligned}
\end{equation}

\vspace{2mm}

\noindent  By using  convexity of $\Omega$ along with the fact that $0$ is a flat point of it, $\exists$ a convex function $$L : \mathbb{R}^{n-1} \longrightarrow \mathbb{R}$$ and  $r_{0} > 0$ such that 
$$Q[r_{0} \times r_{0}] \cap \partial \Omega = \Big\{(x',\; x_{n})\;\;: \;\; x_{n} = L(x'), \; x' \in T_{r_{0}}\Big\}.$$

\vspace{2mm}

\noindent Let $r$ be any non-negative real number such that $r \leq r_{0}.$ Let us define 
$$\tilde{L}(r) = \max\bigg\{\frac{L(x')}{|x'|}\;\;:\;\; x' \in T_{r}\bigg\}.$$
\noindent With  $G_{0}(r)$ obtained from \eqref{12May1}, we define 
$$\tilde{G_{0}}(r) = \sup \big\{G_{0}(t)\;\;:\;\; 0 \leq t \leq r \big\}.$$
\noindent Finally, we define
$$\psi(r) = \max \big\{\tilde{L}(r),\; \tilde{G_{0}}(r)\big\}.$$ 

\vspace{2mm}

\noindent From the observation that if $\tilde{L}(r) = 0,$ then 
$$\partial\Big(\Omega \cap Q[\tilde{r} \times \tilde{r}]\Big) \; \subset\; \partial \mathbb{R}_{+}^{n}\;\; \mathrm{for\; some}\; \tilde{r} < r_{0}.$$
Therefore, from Lemma \ref{MainLemma}, $w$ is differentiable at $0.$ 

\vspace{2mm}

\noindent Finally, we can assume that $$\tilde{L}(r) > 0 \;\; \mathrm{in}\;\; (0,\; r_{0}).$$
Set
$$ \sigma_{j} \;=\; \frac{r_{0}\;\psi(r_{0})}{2^{j}}. $$
\vspace{2mm}

\noindent From \cite[Equation\;(3.4)]{LiandWang1}, for each $j$, the following equation 
\begin{equation}\label{13May1}
r_{j} \; \sqrt{\psi(r_{j})} = \sigma_{j}
\end{equation}
has unique solution such that 
\begin{equation}\label{13May2}
r_{j} > r_{j+1} > 0, \;\; r_{j}\; \rightarrow \; 0\;\; \mathrm{as}\; j \rightarrow \infty.
\end{equation}

\noindent By reducing $r_{0}$ suitably, without loss of generality, we can assume that $\sqrt{\psi(r_{0})} \leq \frac{1}{\Lambda}.$

\vspace{2mm}

\noindent \textbf{Claim I:} For each $j,$ the following estimate holds$:$
\begin{equation}\label{14May1}
0 \;\;\leq\;\; w_{j}(x) - w(x) \;\;\leq\; \;C \;r_{j}\; \psi(r_{j})\;\; \mathrm{for \; each}\; x \in Q[r_{j} \times r_{j}] \cap \Omega.
\end{equation}
\noindent Indeed, by applying the generalized maximum principle \cite{Escauriaza,CaffarelliCrandallKocanSwiech} for $q \in (n - n_{0},\; n),$
where $n_{0}$ is a small universal constant depends only on $n$ and ellipticity constants, we have
\begin{equation}\label{14May2}
0 \;\; \leq \;\; w(x) \;\; \leq \;\; w_{j}(x) \;\; \mathrm{for \; each}\; x \in Q[r_{j} \times r_{j}] \cap \Omega.
\end{equation}
\noindent Using above estimate \eqref{14May2} and by applying  once again the generalized maximum principle \cite{Escauriaza,CaffarelliCrandallKocanSwiech} for $q \in (n - n_{0},\; n),$ we have
\begin{equation}\label{14May3}
0 \;\; \leq \;\; w(x) \;\; \leq \;\; w_{j + 1}(x) \;\;\leq \;\; w_{j}(x)\;\; \mathrm{for \; each}\; x \in Q[r_{j + 1} \times r_{j + 1}] \cap \Omega.
\end{equation}
From \eqref{12May2} and definition of $\tilde{L},\; \tilde{G}_{0},\; \psi$, we have the following estimate for each $x \in Q[r_{j} \times r_{j}] \cap \partial \Omega$: 
\begin{equation}\label{14May4}
w_{0}(x)\; \leq \;a_{0} x_{n}\; +\; C \tilde{G_{0}}(r_{j})r_{j}\; \leq \;a_{0}r_{j}h(r_{j})\; + \;C \tilde{G_{0}}(r_{j})r_{j}\; \leq \;C \psi(r_{j})r_{j}.
\end{equation}
From \eqref{14May3} and \eqref{14May4}, we get
$$w_{j}(x) - w(x) \;\; \leq \;C \psi(r_{j})r_{j}\;\; \mathrm{on}\; Q[r_{j} \times r_{j}] \cap \partial \Omega. $$
Now, generalized maximum principle \cite{Escauriaza,CaffarelliCrandallKocanSwiech} for $q \in (n - n_{0},\; n),$ gives  
\begin{equation}\label{14May5}
w_{j}(x) - w(x) \;\; \leq \;C \psi(r_{j})r_{j}\;\; \mathrm{in}\; Q[r_{j} \times r_{j}] \cap \Omega. 
\end{equation}

\noindent Combining \eqref{14May2} and \eqref{14May5}, observation follows.

\vspace{2mm}

\noindent \textbf{Claim II:} 

\vspace{2mm}

\noindent $\exists$ a sequence $C_{j}$ such that
\begin{equation}\label{14May6}
|w_{j}(x) - a_{j}x_{n}|\;\; \leq \;\; C_{j}\sigma_{j}\;\;\mathrm{in}\; Q[r_{j} \times r_{j}] \cap \Omega, 
\end{equation}
and $C_{j} \rightarrow 0$ as $j \rightarrow \infty.$

\vspace{2mm}

\noindent Indeed, let $C_{j} = G_{j}(\sigma_{j}).$
Then \eqref{14May6} follows from \eqref{12May1}. 
 Also using similar lines of proof as in step III of \cite[Proof of Theorem 1.1]{LiandWang1}, we have 
$$C_{j} \;\rightarrow\; 0\;\; \mathrm{as}\; j \rightarrow \infty.$$

\vspace{2mm}

\noindent \textbf{Claim III:}\;\; $w$ is differentiable at $0.$

\vspace{2mm}

\noindent  Indeed, from \eqref{14May3} and $a_{j} := \frac{\partial w_{j}(0)}{\partial x_{n}},$ we see that $$a_{j} > a_{j+1},\;\;\mathrm{for}\; j = 0, 1, 2,\cdot\cdot\cdot.$$ 
So let us suppose that 
$$\lim_{j\rightarrow \infty}\; a_{j} := a.$$

\noindent Using the Claim II, Claim III and \eqref{13May1}, we get the following estimates for each $x \in Q[r_{j} \times r_{j}] \cap \Omega:$
\begin{equation}\label{15May1}
\begin{aligned}
|w(x) - a x_{n}|\;\; &\leq\;\; |w(x) - w_{j}(x)| \;+ \;|w_{j}(x) - a_{j}x_{n}| \;+ \;|a_{j}x_{n} - a x_{n}|\\
\;\;&\leq\;\; C r_{j} \psi(r_{j})\; + \; C_{j}\sigma_{j} \; + \; |a - a_{j}|\sigma_{j}\\
\;\;&=\;\; \Big(C  \sqrt{\psi(r_{j})}\; + \; C_{j} \; + \; |a - a_{j}|\Big)\sigma_{j}.
\end{aligned}
\end{equation}
 
 \noindent Let us choose $j \geq 0$ with $\sigma_{j + 1} < r \leq \sigma_{j}$ for $r \in (0,\;\sigma_{0}).$\; Set
 
 $$ C(r) := \sup\;\Big\{C \sqrt{\psi(r_{i})}\; + C_{i}\; + \; |a - a_{i}|\; : \; j \leq i\Big\}.$$

 \noindent Using \eqref{15May1} and $\frac{\sigma_{j}}{r} \leq \frac{\sigma_{j}}{\sigma_{j +1}} = 2$, finally we get the following estimates:
$$|w(x) - a x_{n}| \leq \Big(\tilde{C} \sqrt{\psi(r_{i})}\; + C_{j}\; + \; |a - a_{j}|\Big)\sigma_{j} \leq 2 C(r)r.$$
 It is easy to see that $C(r) \rightarrow 0$ as $r \rightarrow 0.$
 \noindent Hence $w$ is differentiable at $0$ with $\nabla w(0) = a e_{n}.$
This completes the proof of the lemma.

%\noindent Using the generalised maximum principle \cite{Escauriaza,CaffarelliCrandallKocanSwiech} instead of maximum principle, the same lines of proof will work as in \cite[ Setion 3, Proof of Theorem 1.1]{Li and Wang}. For the conciseness, we skip the details.

\section{Proof of the Theorem \ref{22JANTHM}}

\noindent By translation, rotation and scaling, without loss of generality, we can assume that $0 \in \partial \Omega$ and only consider the Log-Lipschitz of $w$ at $0.$

\vspace{2mm}

\noindent By \cite[Theorem 1.1]{LiandWang1}, $v$ is differentiable at the origin and $ \nabla v(0) = a\; e_{n}$ for some constant $a$. i.e., 
\begin{align*}
|v - a \;x_{n}| &\leq C(|x|)\;|x|,
\end{align*}
where $C(|x|) \rightarrow 0$ as $|x| \rightarrow 0.$

\vspace{2mm}

\noindent This implies that $\exists$ $\lambda > 0$ such that
\begin{align*}
|v - a \;x_{n}| &\leq \frac{|x|}{2},\;\;\mathrm{whenever}\; |x| \leq \lambda.
\end{align*}

\vspace{2mm}

\noindent Now, this implies that

\begin{equation}\label{SAT23JAN1}
\|v - a \; x_{n}\|_{L^{\infty}(B_{\lambda}\; \cap\; \Omega)} \leq \frac{\lambda}{2}.
\end{equation}

\noindent Estimates showed for $v$ implies that

$$\|w - a \; x_{n}\|_{L^{\infty}(B_{\lambda}\; \cap\; \Omega)} \leq \lambda.$$

\noindent Indeed,

\begin{align*}
\|w - a \;x_{n}\|_{L^{\infty}(B_{\lambda}\; \cap\; \Omega)} &= \|w - v\|_{L^{\infty}(B_{\lambda}\; \cap\; \Omega)} + \|v - a \;x_{n}\|_{L^{\infty}(B_{\lambda}\; \cap\; \Omega)}\\
& \leq \varepsilon + \frac{\lambda}{2}\;\;\;\;\;\;\;\;\Big(\text{using}\; \eqref{SAT23JAN0} \; \text{and}\; \eqref{SAT23JAN1}\Big),\\
& \leq \lambda \;\;\;\;\;\;\;\;\;\;\;\;\;\;\;\Big(\text{by choosing  $ \varepsilon = \frac{\lambda}{2}$, which in turn fixes $\delta$ }\Big).
\end{align*}

\noindent We shall establish by induction process the existence of a sequence of polynomials 

$$L_{k}(x) = k\; a\; x_{n}$$
%satisfying 
%\begin{equation}\label{SAT23JAN2}
%|L_{k}| \leq C \;k \;\lambda^{k}\;\;\; \text{for some positive constant $C$, }
%\end{equation}
such that 
\begin{equation}\label{SAT23JAN3}
\|w - L_{k} \;x\|_{L^{\infty}(B_{\lambda^{k}}\; \cap\; \Omega)} \leq \lambda^{k}.
\end{equation}

\noindent The case $k = 1$ is precisely \eqref{SAT23JAN1}. Now, suppose that we have proved the $k$-th step of induction.  

\noindent Define 

$$ u(x) = \frac{(w -  L_{k})(\lambda^{k} x)}{\lambda^{k}} \;\; \; \text{in}\; \; \Omega_{\lambda^{k}} \cap B_{1},$$

\noindent where $\Omega_{\lambda^{k}} = \{ x \in \Omega \;\;:\;\; \;\lambda^{k} \;x \in \Omega\}.$

\vspace{2mm}

\noindent Then, it is easy to verifies that $w$ is a strong solution to 

$$ -b_{ij}(\lambda^{k} x) \frac{\partial^{2} u}{\partial x_{i}  \partial x_{j}}\; =\; \lambda^{k} \;g(\lambda^{k} x) \;\;\;\mathrm{in}\;\;\Omega_{\lambda^{k}} \cap B_{1}.$$

\noindent Set $$\tilde{g}(x) \;:= \;\lambda^{k} \;g(\lambda^{k} x),$$ then it is easy to see that $\|\tilde{g}\|_{L^{n}(\Omega)} \leq \delta.$ 

\vspace{2mm}

\noindent Therefore, by induction assumption, $\exists$ an affine function 
$$\tilde{L}(x) \;=\; a\; x_{n}\;\; \;\mathrm{in}\;\; \Omega_{\lambda^{k}} \cap B_{\lambda},$$
such that 
$$\|u - \tilde{L}\|_{(\Omega_{\lambda^{k}} \cap B_{\lambda})} \leq \lambda.$$
If we define 
\begin{equation}\label{18May1}
L_{k +1}(y)\; :=\; L_{k}(y) + \lambda^{k}\;\tilde{L}(\frac{y}{\lambda^{k}})\;\;\; \mathrm{for}\;\; y \in \Omega \cap B_{\lambda^{k + 1}},
\end{equation}

\noindent Then, this implies that 
$$\|w - L_{k + 1}\|_{L^{\infty}(B_{\lambda^{k + 1}}\; \cap\; \Omega)} \leq \lambda^{k + 1}.$$

\noindent Moreover, it follows from \eqref{18May1} that 

\begin{equation}\label{19May1}
L_{k + 1}(y)\; = \; (k + 1)\;a \;y_{n}\; \;\mathrm{for}\;\; y \in \Omega \cap B_{\lambda^{k + 1}}.
\end{equation}
\noindent Finally, in view of \eqref{SAT23JAN3} and \eqref{19May1}, we get
\begin{align*}
\|w\|_{B_{\lambda^{k}}} &\leq  \|w - L_{k} \;x\|_{L^{\infty}(B_{\lambda^{k}}\; \cap\; \Omega)} + \|L_{k}\|_{L^{\infty}(B_{\lambda^{k}}\; \cap\; \Omega)} \\
& \leq \lambda^{k} + |a| \; k \; \lambda^{k}\\
& \leq C\; k\; \lambda^{k},
\end{align*}
where $C = C(a) = C(v) = C\big(\|f\|_{L^{n}(\Omega)},\;\|u\|_{L^{\infty}(\Omega)}\big)$ is a positive constant.

\section{Acknowledgement} 

The author would like to thank  Agnid Banerjee for  discussions and  suggestions concerning the  preparation of the manuscript.  The author is also grateful to TIFR CAM for the  financial support.

 \bibliographystyle{alpha}

\end{document}